\documentclass{amsart}
\newtheorem{lemma}{Lemma}

\newcommand{\qu}{\mathbb{Q}}
\newcommand{\zet}{\mathbb{Z}}
\newcommand{\ce}{\mathbb{C}}
\newtheorem{theorem}{Theorem}
\newtheorem{corollary}{Corollary}
\begin{document}
\title{Probabilistic Galois Theory}
\subjclass[2000]{11C08, 11G35, 11R32, 11R45}
\author{Rainer Dietmann}
\address{Department of Mathematics, Royal Holloway, University of London\\
TW20 0EX Egham, United Kingdom}
\email{Rainer.Dietmann@rhul.ac.uk}
\begin{abstract}
We show that there are at most $O_{n,\epsilon}(H^{n-2+\sqrt{2}+\epsilon})$
monic integer
polynomials of degree $n$ having height at most $H$ and Galois group different
from the full symmetric group $S_n$, improving on the previous 1973 world record
$O_{n}(H^{n-1/2}\log H)$.
\end{abstract}
\thanks{This work has been supported by grant EP/I018824/1 `Forms in many
variables'.}
\maketitle
\section{Introduction}
Given a `random' monic integer polynomial of degree $n$, one should expect its
Galois group to be the full symmetric group $S_n$ with probability one. This
has been confirmed by van der Waerden (\cite{W}), even in a quantitative
form which is our main concern in this paper. To be more precise, let
\begin{eqnarray*}
  E_n(H) & = & \#\{ (a_1, \ldots, a_n) \in \zet^n : |a_i| \le H
  \; (1 \le i \le n) \mbox{ and }\\
  & & \mbox{the splitting field of } X^n+a_1X^{n-1}+\ldots+a_n \\
  & & \mbox{over $\qu$ does not have Galois group $S_n$}\}.
\end{eqnarray*}
Then van der Waerden, using reductions modulo $p$ and an elementary sieve argument,
shewed that
\[
  E_n(H) \ll_n H^{n-\frac{1}{6(n-2)\log \log H}}.
\]
Later, Knobloch (\cite{K1}, \cite{K2}) improved this to
\[
  E_n(H) \ll_n H^{n-\frac{1}{18n(n!)^3}},
\]
and Gallagher (\cite{G}), applying the large sieve to van der Waerden's argument,
obtained
\[
  E_n(H) \ll_n H^{n-\frac{1}{2}} \log H.
\]
Apart from $n \le 4$ (see \cite{L}, \cite{D0}), where the conjectured exponent
$n-1+\epsilon$ has been confirmed, and Zywina's recent refinement
(\cite{Z}, Proposition 1.5)
\[
  E_n(H) \ll_n H^{n-\frac{1}{2}}
\]
for sufficiently large $n$, this has been the record for the last
40 years. It is our aim to establish the following improvement in this paper.
\begin{theorem}
\label{piigs}
Let $n \ge 3$, and let $\epsilon>0$. Then
\[
  E_n(H) \ll_{n,\epsilon} H^{n-2+\sqrt{2}+\epsilon}.
\]
\end{theorem}
In contrast to Gallagher's approach based on sieve methods, we rely on
Galois resolvents and recent advances on bounding
the number of integral points on curves or surfaces. In fact, using these
methods in \cite{D} we could show that if $G$ is a subgroup of $S_n$, then
\begin{eqnarray}
\label{suedtirol}
 & & \#\{ (a_1, \ldots, a_n) \in \zet^n : |a_i| \le H
  \; (1 \le i \le n) \mbox{ and } \nonumber \\
  & & X^n+a_1X^{n-1}+\ldots+a_n \mbox{ has Galois group $G$}\}
  \ll_{n,\epsilon} H^{n-1+\delta_G+\epsilon}, \\ \nonumber
\end{eqnarray}
where
\[
  \delta_G = \frac{1}{|S_n/G|},
\]
and $|S_n/G|$ is the index of $G$ in $S_n$. Now it is well known
(see for example Chapter 5.2 in \cite{DM}),
that if $G$ is a subgroup of $S_n$ different from $S_n$
and $A_n$, then $|S_n/A_n| \ge n$. Hence Theorem \ref{piigs} follows from our
previous result (\ref{suedtirol}) and the following improved bound for polynomials
having alternating Galois group.
\begin{theorem}
\label{nachrichten}
Let $\epsilon>0$. Then
\begin{eqnarray*}
 & & \#\{ (a_1, \ldots, a_n) \in \zet^n : |a_i| \le H
  \; (1 \le i \le n) \mbox{ and } \nonumber \\
  & & X^n+a_1X^{n-1}+\ldots+a_n \mbox{ has Galois group $A_n$}\}
  \ll_{n,\epsilon} H^{n-2+\sqrt{2}+\epsilon}. \\ \nonumber
\end{eqnarray*}
\end{theorem}
The new tool available for the proof of Theorem \ref{nachrichten} is a recent result
by Salberger \cite{S} which allows us to bound the number of integer zeros on
surfaces rather than curves. It is important for our application that this can be
done in rather `lopsided' boxes.
The main difficulty then is to show that there are no lines in the surface under
consideration.\\ \\
\textit{Acknowledgment.} The author would like to thank Dr T.D. Browning for
bringing the references \cite{B} and \cite{S} to his attention.
\section{Semi-explicit discriminant formulas}
\label{economists}
\label{bad_teacher}
In this section we establish some useful properties of the discriminant.
We start off with a result on the maximum size of the zeros of a
complex polynomial.
\begin{lemma}
\label{kapstadt}
Let $f(X)=a_0X^n+a_1X^{n-1}+\ldots+a_n \in \ce[X]$. Then all roots
$z \in \ce$ of the equation $f(z)=0$ satisfy the inequality
\[
|z| \le \frac{1}{\sqrt[n]{2}-1} \cdot \max_{1 \le k \le n}
\sqrt[k]{\left| \frac{a_{k}}{a_0 {n \choose k}} \right|}.
\]
\end{lemma}
\begin{proof}
This is Theorem 3 in \S27 of \cite{M}.
\end{proof}
For a monic polynomial $f(X)=X^n+a_1 X^{n-1} + \ldots + a_n \in \ce[X]$ with
roots $\alpha_1, \ldots, \alpha_n \in \ce$, write
\[
  \Delta=\Delta(a_1, \ldots, a_n)
\]
for its discriminant
\begin{equation}
\label{sonne}
  \Delta(a_1, \ldots, a_n) = \prod_{i<j} (\alpha_i-\alpha_j)^2.
\end{equation}
As is well known, $\Delta(a_1, \ldots, a_n)$ is a polynomial in
$a_1, \ldots, a_n$ having integer coefficients.
For trinomials, the discriminant takes a particularly easy shape.
\begin{lemma}
\label{regen}
Let $n \ge 2$. Then the polynomial $X^n+pX+q$ has discriminant
\[
  (-1)^{\frac{n(n-1)}{2}} n^n q^{n-1} + (-1)^{\frac{(n-1)(n-2)}{2}}
  (n-1)^{n-1} p^n.
\]
\end{lemma}
\begin{proof}
This is a well known result; see for example exercise 35 on page 621 in
\cite{DF}.
\end{proof}
\begin{lemma}
\label{cameron_diaz}
Let $n \ge 2$.
In the notation from above, for fixed $a_1, \ldots, a_{n-1}$ consider
$\Delta(a_n)=\Delta(a_1, \ldots, a_n)$ as a polynomial in $a_n$. Then
\[
  \Delta(a_n) = (-1)^{\frac{n(n-1)}{2}} n^n a_n^{n-1} + O(a_n^{n-2}).
\]
\end{lemma}
\begin{proof}
Choosing $a_1=\ldots=a_{n-1}=0$, Lemma \ref{regen} shows that the monomial\\
$(-1)^{\frac{n(n-1)}{2}} n^n a_n^{n-1}$ indeed occurs. To show that
for fixed $a_1, \ldots, a_{n-1}$ all other terms are of order $a_n^{n-2}$
or smaller in $a_n$, let us suppose the contrary: Then
\begin{equation}
\label{california}
  \Delta(a_n) = (-1)^{\frac{n(n-1)}{2}} n^n a_n^{n-1} +
  f(a_1, \ldots, a_{n-1}) a_n^{\alpha} + O(a_n^{n-2}),
\end{equation}
where $\alpha \ge n-1$ and
$f$ is an integer polynomial in $a_1, \ldots, a_{n-1}$,
not vanishing identically in $a_1, \ldots, a_{n-1}$. Lemma \ref{regen}
shows that $f$ cannot
be identically a constant. Let $\epsilon>0$ be sufficiently small,
and let $H$ be sufficiently large in terms of $\epsilon$.
Now if $a_1, \ldots, a_n \in \zet$ such that
\begin{equation}
\label{rod}
  |a_i| \asymp H^\epsilon \;\;\; (1 \le i \le n-1)
\end{equation}
and
\begin{equation}
\label{taylor}
  |a_n| \asymp H,
\end{equation}
then by Lemma \ref{kapstadt} with $a_0=1$,
all roots $\alpha_1, \ldots, \alpha_n$ of $f$ satisfy
\begin{equation}
\label{voegel}
  |\alpha_i| \ll H^{1/n} \;\;\; (1 \le i \le n).
\end{equation}
Now by (\ref{sonne}) and (\ref{voegel}), we have
\begin{equation}
\label{hitch}
  |\Delta| \ll H^{\frac{2}{n} \cdot \frac{n(n-1)}{2}} \ll H^{n-1}.
\end{equation}
By (\ref{california}), the assumption $\alpha \ge n-1$
and our observation on $f$ above, it is certainly possible to choose
$a_1, \ldots, a_n \in \zet$ satisfying (\ref{rod}) and (\ref{taylor})
such that
\begin{equation}
\label{tippi_hedren}
  |\Delta| \gg H^{n-1+\epsilon}.
\end{equation}
Since inequalities (\ref{hitch}) and (\ref{tippi_hedren}) are
inconsistent, we reached a contradiction.
\end{proof}
\begin{lemma}
\label{bonus}
Let $n \ge 2$.
In the notation from above, for fixed $a_1, \ldots, a_{n-2}$ consider
$\Delta(a_{n-1},a_n)=\Delta(a_1, \ldots, a_n)$ as a polynomial in
$a_{n-1}$ and $a_n$. Then
\[
  \Delta(a_{n-1}, a_n) = (-1)^{\frac{(n-1)(n-2)}{2}}
  (n-1)^{n-1} a_{n-1}^n + \Phi(a_{n-1}, a_n),
\]
where $\Phi$ is an integer polynomial in $a_{n-1}$ and $a_n$ of degree
strictly less than $n$, i.e. in all monomials $a_{n-1}^\alpha a_n^\beta$
occurring in $\Phi$, we have $\alpha+\beta<n$.
\end{lemma}
\begin{proof}
Choosing $a_1=\ldots=a_{n-2}=a_n=0$, Lemma \ref{regen} shows that the monomial
$(-1)^{\frac{(n-1)(n-2)}{2}} (n-1)^{n-1} a_{n-1}^n$ indeed occurs in
$\Delta$. To show that for fixed $a_1, \ldots, a_{n-2}$ the polynomial
$\Phi$ has degree strictly less than $n$, let us assume the contrary:
Then in $\Delta(a_1, \ldots, a_n)$ at least one monomial of the form\\
$f(a_1, \ldots, a_{n-2})a_{n-1}^\alpha a_n^\beta$
different from $(-1)^{(n-1)(n-2)/2}(n-1)^{n-1}a_{n-1}^n$
must occur, where $f$
is an integer polynomial in $a_1, \ldots, a_{n-2}$, not vanishing identically,
and $\alpha+\beta \ge n$. Lemma \ref{regen} shows that $f$ cannot be
identically a constant. Let $\epsilon>0$ be sufficiently small, and let
$H$ be sufficiently large in terms of $\epsilon$.
Now if $a_1, \ldots, a_n \in \zet$ such that
\begin{equation}
\label{rod2}
  |a_i| \asymp H^\epsilon \;\;\; (1 \le i \le n-2)
\end{equation}
and
\begin{equation}
\label{taylor2}
  |a_{n-1}|, |a_n| \asymp H,
\end{equation}
then by Lemma \ref{kapstadt} with $a_0=1$,
all roots $\alpha_1, \ldots, \alpha_n$ of $f$ satisfy
\begin{equation}
\label{voegel2}
  |\alpha_i| \ll H^{\frac{1}{n-1}} \;\;\; (1 \le i \le n).
\end{equation}
Now by (\ref{sonne}) and (\ref{voegel2}), we have
\begin{equation}
\label{hitch2}
  |\Delta| \ll H^{\frac{2}{n-1} \cdot \frac{n(n-1)}{2}} \ll H^{n}.
\end{equation}
By our observation above ($\Delta(a_1, \ldots, a_n)$ contains a term
$f(a_1, \ldots, a_{n-2})a_{n-1}^\alpha a_n^\beta$ where $\alpha+\beta \ge n$
and $f$ is not identically a constant), it is certainly possible to choose
$a_1, \ldots, a_n \in \zet$ satisfying (\ref{rod2}) and (\ref{taylor2})
such that
\begin{equation}
\label{tippi_hedren2}
  |\Delta| \gg H^{n+\epsilon}.
\end{equation}
Since inequalities (\ref{hitch2}) and (\ref{tippi_hedren2}) are inconsistent,
we reached a contradiction.
\end{proof}
\section{Lines in the discriminant variety}
Our goal in this section is to show that, in the notation of \S
\ref{bad_teacher}, for fixed $a_1, \ldots, a_{n-2}$, the intersection of the
discriminant variety $z^2=\Delta(a_{n-1},a_n)$ with a line
$d_1 a_{n-1}+d_2 a_n+d_3=0$ has only few integer points.
\begin{lemma}
\label{vilnius}
Let $n \ge 3$,
let $a_1, \ldots, a_{n-2} \in \zet$ and $(c_1,c_2) \in \qu^2$. Then, in
the notation of \S \ref{bad_teacher}, the polynomial
\[
  z^2-\Delta(a_1,\ldots,a_{n-2},c_1a_n+c_2,a_n)
\]
as a polynomial in $z$ and $a_n$ is irreducible over $\qu$.
\end{lemma}
\begin{proof}
Write $\Delta(a_n)=\Delta(a_1,\ldots,a_{n-2},c_1 a_n+c_2, a_n)$. We have to
show that $z^2-\Delta(a_n)$ is irreducible over the rationals. If this were
not true, then necessarily $\Delta(a_n) \equiv f(a_n)^2$ identically in $a_n$,
for a rational polynomial $f$. In particular, the term in $\Delta(a_n)$
having biggest exponent in $a_n$ must be of the form $c^2a_n^{2k}$, for a
rational non-zero
$c$ and a non-negative integer $k$. Let us first suppose that
$c_1 \ne 0$. Then Lemma \ref{bonus} shows that the term in $\Delta(a_n)$
with biggest exponent is
\[
  (-1)^{\frac{(n-1)(n-2)}{2}} (n-1)^{n-1} c_1^n a_n^n.
\]
For odd $n$ it is obvious that this can't be of the form $c^2 a_n^{2k}$.
For even $n \ge 4$ indeed $c_1^n a_n^n$ is a square, but
$|(-1)^{(n-1)(n-2)/2}|=1$ and $(n-1)^{n-1}$ is an odd power, hence no square
of a rational number. Hence again the expression can't be of the form
$c^2 a_n^{2k}$.
In case of $c_1=0$, by Lemma \ref{cameron_diaz} the term in $\Delta(a_n)$
having biggest exponent is
\[
  (-1)^{\frac{n(n-1)}{2}} n^n a_n^{n-1}.
\]
Again, analogously to above it is easily verified that this expression
can't be a square.
Thus $\Delta(a_n) \equiv f(a_n)^2$ is impossible, and
$z^2-\Delta(a_n)$ must be irreducible over the rationals.
\end{proof}
\begin{lemma}
\label{heathrow}
Let $n \ge 3$,
let $a_1, \ldots, a_{n-2} \in \zet$ and $c \in \qu$. Then, in the notation
of \S \ref{bad_teacher}, the polynomial
\[
  z^2-\Delta(a_1,\ldots,a_{n-1},c)
\]
as a polynomial in $z$ and $a_{n-1}$ is irreducible over $\qu$.
\end{lemma}
\begin{proof}
Similarly to the proof of Lemma \ref{vilnius}, Lemma \ref{bonus} shows that the
term in $\Delta(a_{n-1})=\Delta(a_1, \ldots, a_{n-2}, a_{n-1}, c)$
with biggest exponent is
\[
  (-1)^{\frac{(n-1)(n-2)}{2}} (n-1)^{n-1} a_{n-1}^n,
\]
which can't be a rational square. This implies that $\Delta(a_{n-1})$ can't be
the square of a rational polynomial, whence $z^2-\Delta(a_{n-1})$ must be
irreducible.
\end{proof}
\begin{lemma}
\label{athen}
Let $F \in \zet[X_1, X_2]$ be of degree $d$
and irreducible over $\qu$. Further,
let $P_1, P_2$ be real numbers such that $P_1, P_2 \ge 1$, and let
\[
  N(F; P_1, P_2) = \#\{\mathbf{x} \in \zet^2:
  F(\mathbf{x})=0
  \mbox{ and } |x_i| \le P_i \; (1 \le i \le 2)\}.
\]
Moreover, let
\[
  T = \max \left\{ \prod_{i=1}^2 P_i^{e_i} \right\}
\]
with the maximum taken over all integer $2$-tuples $(e_1, e_2)$ for
which the corresponding monomial $X_1^{e_1} X_2^{e_2}$ occurs
in $F(X_1, X_2)$ with nonzero coefficient. Then
\begin{equation}
\label{pigs}
  N(F; P_1, P_2) \ll_{d, \epsilon}
  \max\{P_1, P_2\}^{\epsilon}
  \exp \left( \frac{\log P_1 \log P_2}{\log T} \right).
\end{equation}
\end{lemma}
\proof
This can be immediately deduced from \cite{BH}, Theorem 1 or \cite{HB},
Theorem 15; see the proof of the same Lemma 8 in \cite{D} for more details.
\begin{corollary}
\label{greenwich}
Let $f(X_1,X_2) \in \qu[X_1,X_2]$ be of degree $d$ and irreducible over $\qu$.
Moreover, let $P \ge 1$ and $\epsilon>0$. Then, uniformly in $f$, we have
\[
  \#\{(x_1,x_2) \in \zet^2: |x_1|, |x_2| \le P \mbox{ and } f(x_1,x_2)=0\}
  \ll_{d,\epsilon} P^{\frac{1}{d}+\epsilon}.
\]
\end{corollary}
\begin{proof}
This follows immediately from Lemma \ref{greenwich} by choosing $P_1=P_2=P$.
In fact, this Corollary is the well known Bombieri-Pila result \cite{BP}.
\end{proof}
\begin{lemma}
\label{contagion}
Let $\epsilon>0$, let $c \ge 1$, let $n \ge 3$,
let $a_1, \ldots, a_{n-2} \in \zet$ and let
$d_1, d_2, d_3 \in \qu$ such that $(d_1,d_2) \ne (0,0)$. Then, in the
notation of \S \ref{bad_teacher}, the system of simultaneous equations
\[
  z^2=\Delta(a_1,\ldots,a_n)
\]
and
\begin{equation}
\label{staines}
  d_1 a_{n-1}+d_2a_n+d_3=0
\end{equation}
has at most $O_{n,\epsilon}(H^{\frac{1}{2}+\epsilon})$ solutions
$z,a_{n-1},a_n$ such that $|a_{n-1}|, |a_n| \le H$ and $|z| \le H^c$.
\end{lemma}
\begin{proof}
First suppose that $d_1 \ne 0$. Then by (\ref{staines}) we can write
$a_{n-1}=c_1a_n+c_2$ for suitable $c_1, c_2 \in \qu$. By Lemma
\ref{vilnius}, the polynomial
\[
  f(z,a_n)=z^2-\Delta(a_1, \ldots, a_{n-2},
  c_1 a_n+c_2, a_n)
\]
then is irreducible over $\qu$.
Applying Lemma \ref{athen} with $P_1=H^c$ and $P_2=H$,
noting that $T \ge H^{2c}$ since the term $z^2$
shows up in $f(z,a_n)$, we obtain
\[
  \#\{(z,a_n) \in \zet^2: |z| \le H^c, |a_n| \le H \mbox{ and } f(z,a_n)=0\}
  \ll_{n, \epsilon} H^{\frac{1}{2}+\epsilon},
\]
confirming the lemma in that case. Now assume that $d_1=0$. Then $d_2 \ne 0$,
so by (\ref{staines}) we have $a_n=c$ for a suitable $c \in \qu$. Using
Lemma \ref{heathrow} this time, again we find that $f(z,a_{n-1})=z^2-
\Delta(a_1,\ldots,a_{n-1},c)$ is irreducible over $\qu$, and analogously
to above the conclusion of the lemma follows from Lemma \ref{athen}.
\end{proof}
\section{Absolute irreducibility of the discriminant variety}
In this section we show that for `most' choices of $a_1, \ldots, a_{n-2}$,
the polynomial $z^2-\Delta(a_{n-1},a_n)$ is absolutely irreducible.
\begin{lemma}
\label{rrr}
Let $n$ be a positive integer and
\[
  f(X_1,X_2,X_3)=\sum c_{ijk} X_1^i X_2^j X_3^k
\]
be a rational polynomial of degree $n$. Then there exists an integer
polynomial $F$ in the coefficients $c_{ijk}$ of $f$ such that $f$ is
absolutely irreducible if and only if $F$ evaluated at the $c_{ijk}$ is
different from zero. The polynomial $F$ depends only on $n$.
\begin{proof}
This is a special case of a well known result; see \cite{N}.
\end{proof}
\end{lemma}
\begin{lemma}
\label{ttt}
Let $K$ be any field of characteristic zero. Then the splitting field of
the polynomial
\[
  X^n+aX+b
\]
over the function field $K(a,b)$ has Galois group $S_n$.
\end{lemma}
\begin{proof}
This is essentially Corollary 2 in \cite{Y} (switching $-a$ to $a$ obviously
does not change the result).
\end{proof}
\begin{lemma}
\label{cs}
Let $N(H)$ be the number of integers $a_1, \ldots, a_{n-2}$ such that
$|a_i| \le H \; (1 \le i \le n-2)$ and the polynomial
\[
  z^2-\Delta(a_1, \ldots, a_{n-2}; a_{n-1}, a_n)=z^2-\Delta(a_{n-1},a_n)
\]
as a rational polynomial in $z, a_{n-1}, a_n$ is not absolutely irreducible.
Then
\[
  N(H) \ll H^{n-1}.
\]
\end{lemma}
\begin{proof}
By Lemma \ref{rrr}, there exists an integer polynomial $F(a_1, \ldots,
a_{n-2})$ with the following property: For fixed $a_1, \ldots, a_{n-2}$,
the polynomial $z^2-\Delta(a_{n-1},a_n)$ is absolutely irreducible if and
only if $F(a_1, \ldots, a_{n-2}) \ne 0$. Hence, with respect to Lemma
\ref{cs}, it is sufficient to show that $F$ is not identically
zero. To this end it is enough to find one specialisation for $a_1, \ldots,
a_{n-2}$ for which  $F(a_1, \ldots, a_{n-2}) \ne 0$. It is easy to see that the
choice $a_1=\ldots=a_{n-2}=0$ works. For suppose that in this case
$F(a_1, \ldots, a_{n-2})=0$. Then $z^2-\Delta(a_{n-1},a_n)$ were reducible
over some algebraic extension $K$ of $\qu$. In particular,
$\Delta(a_{n-1},a_n)$ were a square over the polynomial ring $K[a_{n-1},
a_n]$. Hence over the function field $K(a_{n-1}, a_n)$, the polynomial
$X^n+a_{n-1}X+a_n$ had a discriminant being a square, implying that its
Galois group were a subgroup of the alternating group $A_n$ rather than
the full symmetric group $S_n$. This however contradicts Lemma
\ref{ttt}.
\end{proof}
\section{Proof of Theorem \ref{nachrichten}} 
Our main tool in proving Theorem \ref{nachrichten} is the following recent
result of Salberger.
\begin{lemma}
\label{evil}
Let $g(X_1,X_2,X_3) \in \zet[X_1,X_2,X_3]$ be absolutely irreducible of
degree $d$, and let $B_1, B_2, B_3 \ge 1$. Write
\[
  S=\{\mathbf{x} \in \zet^3:g(x_1,x_2,x_3)=0 \mbox{ and }
  |x_i| \le B_i \; (1 \le i \le 3)\}.
\]
Moreover, let
\[
  T = \max\{B_1^{e_1} B_2^{e_2} B_3^{e_3}\},
\]
where the maximum is over all tuples $(e_1, e_2, e_3)$ for which the
corresponding monomial $X_1^{e_1} X_2^{e_2} X_3^{e_3}$ occurs in $g$
with non-zero coefficient. Furthermore, let
\[
  V = \exp \left( \left( \frac{(\log B_1)(\log B_2)(\log B_3)}{\log T}\right)
  ^{\frac{1}{2}}\right).
\]
Finally, let $\epsilon>0$, and write
\[
  B=\max\{B_1, B_2, B_3\}.
\]
Then there exist polynomials $g_1, \ldots, g_J \in \zet[X_1, X_2, X_3]$ and
a finite subset $Z \subset S$ with the following properties:
\begin{itemize}
  \item[(i)] $J \ll_{d,\epsilon} VB^{\epsilon}$,
  \item[(ii)] Each $g_j$ is coprime to $g$ and has degree only bounded in
  terms of $d$ and $\epsilon$,
  \item[(iii)] $\#Z \ll_{d,\epsilon} V^2B^{\epsilon}$,
  \item[(iv)] Each $(x_1,x_2,x_3) \in S\backslash Z$ satisfies
  $g(x_1,x_2,x_3)=g_j(x_1,x_2,x_3)=0$ for some $j \le J$.
\end{itemize}
\end{lemma}
\begin{proof}
This is Lemma 1 in \cite{B}, which in turn is the special case $n=3$ of a
result of Salberger \cite{S}.
\end{proof}
We are now in a position to prove Theorem \ref{nachrichten}.
In the notation of section
\ref{economists}, let
\begin{eqnarray*}
  M(H) & = & \#\{(a_1, \ldots, a_n) \in \zet^n:|a_i| \le H \; (1 \le i \le n)
  \\ & & \mbox{ and } \Delta(a_1, \ldots, a_n) \mbox{ is a rational square}\}.
\end{eqnarray*}
Using the well known criterion that a polynomial has a Galois group
contained in the alternating group if and only if its discriminant is a square,
we conclude that with respect to Theorem \ref{nachrichten}
it is enough to show that
\[
  M(H) \ll_{n, \epsilon} H^{n-2+\sqrt{2}+\epsilon}.
\]
From Lemma \ref{kapstadt} it is clear that there exists a positive constant
$c \ge 1$, such that whenever $z^2=\Delta(a_1, \ldots, a_n)$ where
$|a_i| \le H \; (1 \le i \le n)$ for sufficiently large $H$, then
$|z| \le H^c$. Moreover, since $\Delta(a_1, \ldots, a_n)$ is an integer for
$a_1, \ldots, a_n \in \zet$, the condition $\Delta(a_1, \ldots, a_n)$ being
a rational square is equivalent to it being an integer square. Thus
\begin{eqnarray*}
  M(H) & \ll & \#\{(a_1, \ldots, a_n, z) \in \zet^{n+1}:|a_i| \le H \;
  (1 \le i \le n), \; |z| \ll H^c \\ & &
  \mbox{ and } z^2=\Delta(a_1, \ldots, a_n)\}.
\end{eqnarray*}
Our strategy is now to fix $a_1, \ldots, a_{n-2}$. There are $O(H^{n-2})$
choices for doing so. By Lemma \ref{cs}, with respect to Theorem
\ref{nachrichten}
we may without loss of generality assume that $z^2-\Delta(a_1, \ldots, a_{n-2};
a_{n-1},a_n)=z^2-\Delta(a_{n-1},a_n)$ is absolutely irreducible as a polynomial
in $z,a_{n-1},a_n$. It is now enough to show that for
\[
  S=\{(a_{n-1},a_n,z) \in \zet^3:|a_{n-1}|,|a_n| \le H, |z| \le H^c
  \mbox{ and } z^2=\Delta(a_{n-1},a_n)\}
\]
the upper bound
\begin{equation}
\label{ppp}
  \# S \ll_{n, \epsilon} H^{\sqrt{2}+\epsilon}
\end{equation}
holds true, uniformly in $a_1, \ldots, a_{n-2}$.
Applying Lemma \ref{evil} with $B_1=B_2=H$ and $B_3=H^c$ we
find that $T \ge H^{2c}$, since the term $z^2$ occurs in $z^2-\Delta(a_{n-1},
a_n)$. Hence
\[
  V = \exp \left( \left( \frac{(\log B_1)(\log B_2)(\log B_3)}{\log T}\right)^
  \frac{1}{2}\right) \le H^{\frac{\sqrt{2}}{2}}.
\]
Now by Lemma \ref{evil}, there exist polynomials $g_1, \ldots, g_J \in
\zet[a_{n-1}, a_n, z]$ and a finite subset $Z \subset S$ such that the
following properties hold true:
\begin{itemize}
  \item[(i)] $J \ll_{n, \epsilon} H^{\frac{\sqrt{2}}{2}+\epsilon}$,
  \item[(ii)] Each $g_j$ is coprime to $z^2-\Delta(a_{n-1},a_n)$
  and has degree only bounded in terms of $n$ and $\epsilon$,
  \item[(iii)] $\#Z \ll_{n, \epsilon} H^{\sqrt{2}+\epsilon}$,
  \item[(iv)] Each $(a_{n-1},a_n,z) \in S\backslash Z$ satisfies
  $g_j(a_{n-1},a_n,z)=0$ for some $j \le J$.
\end{itemize}
With respect to (\ref{ppp}), by (iii) it is now sufficient to show that
\begin{equation}
\label{klassentreffen}
  \#(S\backslash Z) \ll_{n, \epsilon} H^{\sqrt{2}+\epsilon}.
\end{equation}
By (i) and (iv), in turn, to show (\ref{klassentreffen}) it is enough
to prove that for every fixed $j \le J$, we have
\begin{eqnarray}
\label{se}
  & & \{(a_{n-1},a_n,z) \in \zet^3:|a_{n-1}|,|a_n| \le H, \; |z| \le H^c,\\
   \nonumber
  & & z^2=\Delta(a_{n-1},a_n) \mbox{ and } g_j(a_{n-1},a_n,z)=0\}
  \ll_{n, \epsilon} H^{\frac{\sqrt{2}}{2}+\epsilon}.
\end{eqnarray}
So fix any $j \le J$ and consider the system of simultaneous equations
\begin{equation}
\label{bs}
  \left\{ \begin{array}{l}
     z^2 = \Delta(a_{n-1},a_n) \\ g_j(a_{n-1},a_n,z) = 0.
  \end{array} \right.
\end{equation}
We are now going to eliminate $z$ from these equations. For each term in
$g_j(a_{n-1},a_n,z)$ containing an even power of $z$ we can just substitute in
a suitable power of $\Delta(a_{n-1},a_n)$. The same way each term in
$g_j(a_{n-1},a_n,z)$ having an odd power of $z$ can be reduced to a term of
the form $z$ times a power of $\Delta(a_{n-1},a_n)$. So we get a system of
simultaneous equations of the form
\begin{equation}
\label{mushroom}
  \left\{ \begin{array}{l}
     z^2 = \Delta(a_{n-1},a_n) \\ zp_j(a_{n-1},a_n)+q_j(a_{n-1},a_n)=0
  \end{array} \right.
\end{equation}
for suitable $p_j, q_j \in \zet[a_{n-1},a_n]$ which is equivalent to
(\ref{bs}), i.e. every solution $(a_{n-1},a_n,z)$ of (\ref{bs})
is also a solution of (\ref{mushroom}) and vice versa. In
particular, the varieties $W_1$ and $W_2$, defined by (\ref{bs}) and
(\ref{mushroom}), respectively, are the same, consequently
also having the same dimension. Since $z^2-\Delta(a_{n-1}, a_n)$ is
absolutely irreducible and coprime
to $g_j$ by property (ii) above, $W_1$ clearly has dimension one, so the same
must be true for $W_2$. Consequently, $p_j$ and $q_j$ cannot both vanish
identically. Thus if $p_j$ vanishes identically, then we are reduced to the system
\[
  \left\{ \begin{array}{l}
     z^2 = \Delta(a_{n-1},a_n) \\ q_j(a_{n-1},a_n)=0
  \end{array} \right.
\]
for a non identically vanishing $q_j$.
Otherwise, we distinguish two cases: For those solutions $(z,a_{n-1},a_n)$ of
(\ref{mushroom}) where $p_j(a_{n-1}, a_n)=0$, we will consider the
system
\[
  \left\{ \begin{array}{l}
     z^2 = \Delta(a_{n-1},a_n) \\ p_j(a_{n-1},a_n)=0,
  \end{array} \right.
\]
where $p_j$ does not vanish identically. For those solutions $(z,a_{n-1},a_n)$ 
of (\ref{mushroom}) where $p_j(a_{n-1}, a_n) \ne 0$,
we can solve
the second equation for $z$ and substitute into the first equation. This way
we are reduced to a system of the form
\[
  \left\{ \begin{array}{l}
     z^2 = \Delta(a_{n-1},a_n) \\ r_j(a_{n-1},a_n)=0,
  \end{array} \right.
\]
for a suitable polynomial $r_j(a_{n-1},a_n) \in \qu[a_{n-1},a_n]$,
namely $r_j(a_{n-1}, a_n) = \Delta p_j^2 - q_j^2$. Now $r_j$
can't be identically zero, since otherwise the identity
\[
  \Delta(a_{n-1},a_n) = \frac{q_j^2(a_{n-1},a_n)}{p_j^2(a_{n-1},a_n)}
\]
would hold true in the function field $\qu(a_{n-1},a_n)$. Since
$\Delta(a_{n-1},a_n) \in \qu[a_{n-1},a_n]$, this immediately implied that
$\Delta(a_{n-1},a_n)$ is not only a square in $\qu(a_{n-1},a_n)$, but even
in the polynomial ring $\qu[a_{n-1},a_n]$. Then, however, the polynomial
$z^2-\Delta(a_{n-1},a_n)$ became reducible over $\qu$, contradicting its
absolute irreducibility. \\
So in all cases we are reduced to bounding the number of solutions of a
system of equations of the form
\begin{equation}
\label{ap}
  \left\{ \begin{array}{l}
     z^2 = \Delta(a_{n-1},a_n) \\ f_j(a_{n-1},a_n)=0
  \end{array} \right.
\end{equation}
for a suitable $f_j \in \qu[a_{n-1},a_n]$ not vanishing identically, subject
to
\begin{equation}
\label{dung}
  |a_{n-1}|, |a_n| \le H \mbox{ and } |z| \le H^c.
\end{equation}
Note that by property (ii) from above and our construction of $f_j$ above, its
degree is bounded in terms of $n$ and $\epsilon$ only. Hence we can factor
$f_j$ over
$\qu$ into $m_j$ irreducible factors $f_{j,i} \; (1 \le i \le m_j)$, where
$m_j$ depends only on $n$ and $\epsilon$. Therefore the number of solutions of
(\ref{ap}) subject to (\ref{dung}) can be bounded by a constant
depending only on $n$ and $\epsilon$,
times the maximal number of solutions of any of the systems
\begin{equation}
\label{dilbert}
  \left\{ \begin{array}{l}
     z^2 = \Delta(a_{n-1},a_n) \\ f_{j,i}(a_{n-1},a_n)=0
  \end{array} \right.
\end{equation}
subject to (\ref{dung}). Since $f_{j,i}(a_{n-1},a_n)$ is irreducible, by Corollary
\ref{greenwich}, the number of solutions of the second equation satisfying
$|a_{n-1}|, |a_n| \le H$ is $O_\epsilon(H^{\frac{1}{d}+\epsilon})$, where
$d$ is the degree of $f_{j,i}$. This is satisfactory for our purposes if $d \ge 2$.
If $d=1$, then $f_{j,i}$ is a linear polynomial, say
\[
  f_{j,i}(a_{n-1},a_n)=d_1^{(j,i)} a_{n-1} + d_2^{(j,i)} a_n + d_3^{(j,i)}
\]
for suitable $d_1^{(j,i)}, d_2^{(j,i)}, d_3^{(j,i)} \in \qu$. Since
$f_j$ does not vanish identically, also its divisor $f_{j,i}$ can't, so not
all of $d_1^{(j,i)}, d_2^{(j,i)}, d_3^{(j,i)}$ can be zero. If
$d_1^{(j,i)}=d_2^{(j,i)}=0$, then necessarily $d_3^{(j,i)} \ne 0$, and
(\ref{dilbert}) has no solution at all. Otherwise, we can invoke Lemma
\ref{contagion} to show that (\ref{dilbert}) has at most
$O_{n,\epsilon}\left(H^{1/2+\epsilon}\right)$ solutions. Hence, in any case,
(\ref{dilbert}) has at most $O_{n,\epsilon}\left(H^{1/2+\epsilon}\right)$
solutions,
and since $m_j$ is bounded in terms of $n$ and $\epsilon$ only, the same is true for
(\ref{ap}). Working backwards through our considerations above,
we find that (\ref{se}) is true with exponent $\frac{1}{2}
+\epsilon$ on the right
hand side, which is even better than claimed. This completes the proof of
\mbox{Theorem \ref{nachrichten}}.

\end{document}